\documentclass[12pt]{article}
\usepackage{amssymb}
\usepackage{amsfonts}
\usepackage{amsmath}
\usepackage[usenames]{color}
\usepackage{mathrsfs}
\usepackage{amsfonts}
\usepackage{amssymb,amsmath}
\usepackage{CJK}
\usepackage{cite}
\usepackage{cases}
\usepackage{amsthm}

\def\cl{\centerline}

\def\vs{\vspace*}

\def\M{\mathscr{W}_{A}(\Gamma)}
\def\Z{\mathbb{Z}}

\def\T{\mathcal{L}}

\def\C{\mathbb{C}}

\def\M{\mathcal{M}}

\def\ni{\noindent}

\numberwithin{equation}{section}
\newtheorem{theo}{Theorem}[section]
\newtheorem{defi}[theo]{Definition}
\newtheorem{coro}[theo]{Corollary}
\newtheorem{lemm}[theo]{Lemma}
\newtheorem{prop}[theo]{Proposition}
\newtheorem{clai}{Claim}
\newtheorem{case}{Case}

\newtheorem{rema}[theo]{Remark}

\begin{document}
\begin{center}
\cl{\large\bf \vs{8pt}A family of simple  non-weight modules over the}
\cl{\large\bf \vs{8pt}  twisted  $N=2$  superconformal algebra}
\cl{ Haibo Chen, Xiansheng Dai and Mingqiang Liu$^{*}$}
\end{center}
\footnote{$^{*}$ Corresponding author.}
{\small
\parskip .005 truein
\baselineskip 3pt \lineskip 3pt

\noindent{{\bf Abstract:}
We construct a class of non-weight modules over the twisted $N=2$ superconformal algebra $\T$.
Let $\mathfrak{h}=\C L_0\oplus\C G_0$  be the Cartan
subalgebra of $\T$, and let $\mathfrak{t}=\C L_0$ be the Cartan
subalgebra of even part      $\T_{\bar 0}$.
These modules over $\T$  when
restricted to the     $\mathfrak{h}$ are free of rank $1$ or when restricted to the     $\mathfrak{t}$ are free of rank $2$.
We provide the sufficient and necessary  conditions for those   modules being simple, as
well as giving the sufficient and necessary  conditions
for two  $\T$-modules being isomorphic.
We also compute the action of  an automorphism  on them.
Moreover, based on  the weighting functor introduced in \cite{N2}, a class of   intermediate series modules $A_t(\sigma)$ are obtained. As a byproduct, we  give a sufficient condition for two $\T$-modules are  not isomorphic.
\vs{5pt}

\ni{\bf Key words:}
Twisted  $N=2$  superconformal algebra, Non-weight module, Simple module.}

\ni{\it Mathematics Subject Classification (2010):} 17B10, 17B65, 17B68.}
\parskip .001 truein\baselineskip 6pt \lineskip 6pt
\section{Introduction}

Throughout this paper,  we denote by $\C,\C^*,\Z$  and $\Z_+$ the sets of complex numbers, nonzero complex numbers, integers and  nonnegative integers respectively. Note that $\frac{1}{2}\Z=(\frac{1}{2}+\Z)\cup\Z$.
   All vector superspaces
(resp. superalgebras, supermodules)   and
 spaces (resp. algebras, modules)    are considered to be over
$\C$.  For a Lie algebra $\mathcal{G}$, we use $\mathcal{U}(\mathcal{G})$ to denote
the universal enveloping algebra. %All modules for Lie superalgebras considered  are $\Z_2$-graded.

The $N=2$ superconformal algebras have four sectors: the Neveu-Schwarz sector,
the Ramond sector, the topological sector  and the twisted sector, which
 was constructed independently by Kac  and by  Ademollo et
al. (see \cite{K,A}).
 It is well known that the first three superconformal algebras are isomorphic to each other, and they are
sometimes called the untwisted sector. Therefore, there are essentially two different
sectors.

The   representation theory  of superconformal algebras  has   attracted
a lot of attention from mathematicians and physicists, which includes  weight  representations
and non-weight  representations.  Some characteristics  related to the structure of Verma modules  over the twisted
 $N=2$ superconformal algebra  were investigated  in \cite{IK,DG}. The classification of indecomposable modules of the intermediate series over the twisted
$N=2$ superconformal algebra   was
achieved in \cite{LSZ}.  All  simple Harish-Chandra modules over the $N=1$ and untwisted $N=2$ superconformal algebras were
classified  respectively in \cite{LPX,S}. In fact, with increasing the number of fermionic currents $N$, the  representations of the superconformal algebras becomes   more and more complicated.

A class of non-weight modules   called  free $\mathcal{U}({\mathfrak{h}})$-modules   were constructed and studied in recent years. Here, ${\mathfrak{h}}$ is  the Cartan subalgebra  (modulo the central elements) and  acts freely.
The free  $\mathcal{U}({\mathfrak{h}})$-modules  were first introduced by Nilsson  in  \cite{N1}  for the complex matrices algebra $\mathfrak{sl}_{n+1}$. In \cite{TZ1}, Tan and Zhao
obtained  them by a very different method.
From then on, the free $\mathcal{U}({\mathfrak{h}})$-modules  were widely investigated over all kinds of  Lie algebras, such as   the finite-dimensional simple Lie algebras \cite{N2} and the infinite-dimensional Lie algebras related to  Virasoro algebras and so on (see, e.g., \cite{WZ,W,LZ,TZ,CG,CY,HCS}).
It is worth mentioning that the free $\mathcal{U}({\mathfrak{h}})$-modules   over some Lie superalgebras  were also  characterized.  In \cite{CZ},    $\mathcal{U}({\mathfrak{h}})$-modules over the basic
Lie superalgebras were investigated, which   showed that $\mathfrak{osp}(1|2n)$ is the only basic Lie superalgebra that
admits such modules. The $\mathcal{U}(\mathfrak{h})$-free modules on some
 infinite-dimensional   Lie superalgebras  such as super-BMS$_3$ algebras,  super-Virasoro algebras,  untwisted $N=2$ super conformal algebras   were  respectively studied  in \cite{YYX1,YYX2,CDLS}.  The aim of the present paper is to classify   such modules for the twisted $N=2$ superconformal algebras, and we further investigate the relevant   twisted modules and   weight modules.

The rest of the paper is organized as follows.

In section $2$, we first recall  the definition of twisted $N=2$ superconformal algebra.  A class
 of non-weight modules over the $\T$ are defined.    Then the simplicities and
isomorphism classes of these modules are respectively  determined in Theorems \ref{th2.3} and \ref{th2ere.4} (also see Theorems \ref{th346782.3} and \ref{th9082.4}).

In Sections $3$ and $4$,    the free $\mathcal{U}(\mathfrak{h})$-modules of rank $1$  over   $\T$ and the free $\mathcal{U}(\mathfrak{t})$-modules of rank $2$  over   $\T$ are respectively  classified in Theorems \ref{th3.5} and  \ref{th9.5}.

 In Section  $5$,
using the automorphism results  in  \cite{FHY}, we   define a class of twisted modules
of $\T$.

At last,   applying the right exact weighting functor introduced  in \cite{N2} to the modules $\mathcal{M}_t(\lambda,\alpha)$,
  we present that the resulting weight modules are exactly  $A_t{(\sigma)}$.

\section{Non-weight modules}
In this section,    we construct  a class of non-weight modules over the twisted  $N=2$ superconformal algebra.
Notice that these modules  are free of rank $1$ when regarded as $\mathcal{U}(\mathfrak{h})$-modules, while these modules are free of rank $2$ when regarded as
$\mathcal{U}(\mathfrak{t})$-modules.

Let us  first recall the definition of  $N=2$ superconformal algebras:
the twisted $N=2$ algebra (i.e. with mixed boundary conditions for the fermionic fields).
\begin{defi}
  {\bf The twisted  $N=2$ superconformal algebra}  $\T$
 is defined as an infinite-dimensional Lie superalgebra over $\C$ with basis
$\{L_m,I_r,G_{p}
\mid m\in \Z, r\in  \frac{1}{2}+\Z, p\in\frac{1}{2}\Z\}$ and  satisfying the following  non-trivial super brackets:
 \begin{equation}\label{def1.1}
\aligned
&[L_m,L_n]= (m-n)L_{m+n},\
[L_m,I_r]= -rI_{m+r},\\&
[L_m,G_p]= (\frac{m}{2}-p)G_{m+p},\ [I_r,G_p]=G_{r+p}, \\&
 [G_p,G_q]= \left\{\begin{array}{llll}
(-1)^{2p}2L_{p+q}&\mbox{if}\
p+q\in\Z,\\[4pt]
(-1)^{2p+1}(p-q)I_{p+q}&\mbox{if}\
p+q\in\frac{1}{2}+\Z,
\end{array}\right.
\endaligned
\end{equation}
  where $m,n\in\Z,$ $r\in  \frac{1}{2}+\Z,$ $p,q\in\frac{1}{2}\Z$.
  \end{defi}
The $\Z_2$-graded of  $\T$ is defined by   $\T=\T_{\bar 0}\oplus \T_{\bar 1}$, where
$\T_{\bar 0}=\{L_m,I_r\mid m\in \Z,r\in\frac{1}{2}+\Z\}$ and $\T_{\bar 1}=\{G_p\mid p\in \frac{1}{2}\Z\}$.
  Clearly,  $\T$ contains
the Witt algebra $\mathcal{V}=\mathrm{span}_{\C}\{L_m
\mid m\in \Z\}$ and the  centerless $N=1$ super-Virasoro algebra.
  Fixing $m\in\Z,\mu\in\C^*$,   we denote
 $$\mathbb{L}_m:=L_m+\mu I_{m-\frac{1}{2}}.$$ % Then we compute that
 %\begin{eqnarray*}
 %[\mathbb{L}_m,\mathbb{L}_n]
 %&=&[L_m+m\mu I_{m-\frac{1}{2}},L_n+n\mu I_{n-\frac{1}{2}}]
 %\\&=&(m-n)L_{m+n}+\big(m(m-\frac{1}{2})-n(n-\frac{1}{2})\big)\mu I_{m+n-\frac{1}{2}}
% \\&=&(m-n)\big(L_{m+n}+(m+n)\mu I_{m+n-\frac{1}{2}}\big)=(m-n)\mathbb{L}_{m+n},
%\end{eqnarray*}
%where $m,n\in\Z,\mu\in\C^*$.
The Lie algebra spanned by  $\{\mathbb{L}_{m}\mid m\in\Z\}$ over $\C$ is isomorphic to the Witt algebra.

Let $\mathcal{G}$ be a Lie superalgebra, and let $V=V_{\bar0}\oplus V_{\bar1}$ be a  $\Z_2$-graded vector space.  Any element $v$ in $V_{\bar0}$ (resp. $v$ in $V_{\bar1}$)
is said to be even   (resp. odd). If
$v$ is even (resp. odd), we  define $|v|=\bar 0$  (resp. $|v|=\bar 1$).  Elements in  $V_{\bar0}$ or  $V_{\bar1}$ are called homogeneous. All  elements in superalgebras and modules are homogenous unless
specified.
A $\mathcal{G}$-module is a $\Z_2$-graded vector space $V$ together
with a bilinear map $\mathcal{G}\times V\rightarrow V$, denoted $(x,v)\mapsto xv$ such that
$$x(yv)-(-1)^{|x||y|}y(xv)=[x,y]v
\quad \mathrm{and}\quad
 \mathcal{G}_{\bar i} V_{\bar j}\subseteq  V_{\bar i+\bar j},$$
where $\bar i,\bar j\in\Z_2,x, y \in \mathcal{G}, v \in V$. Thus there is a parity-change functor $\Pi$ on the category of
$\mathcal{G}$-modules to itself, namely, for any module
$V=V_{\bar0}\oplus V_{\bar1}$, we have $\Pi(V_{\bar 0})=V_{\bar 1}$ and $\Pi(V_{\bar 1})=V_{\bar0 }$.

Take $\lambda\in\C^*,\alpha\in\C$.
 Let $\Omega(\lambda,\alpha)=\C[t]$
as a vector space and  define the  action of Witt algebra as follows:
\begin{eqnarray*}
&&L_m(f(t))=\lambda^m(t+m\alpha)f(t+m),
\end{eqnarray*}
where $m\in\Z,f(t)\in\C[t].$

The following results   were given  in  \cite{LZ,TZ}.
\begin{theo}\label{le2.1}
For $\lambda\in\C^*,\alpha\in\C$,
then
\begin{itemize}
\item[\rm(1)] $\Omega(\lambda,\alpha)$ is a $\mathcal{V}$-module;
\item[\rm(2)] $\Omega(\lambda,\alpha)$ is simple if and only if  $\alpha\neq0$;
\item[\rm(3)] $t\Omega(\lambda,0)\cong\Omega(\lambda,1)$  is the unique
simple proper submodule of $\Omega(\lambda,0)$ with codimension $1$;
\item[\rm(4)] Any free $\mathcal{U}(\C L_0)$-module  of rank $1$ over the Witt algebra  is isomorphic to
some  $\Omega(\lambda,\alpha)$  for $\lambda\in\C^*,\alpha\in\C$.
\end{itemize}
\end{theo}

\subsection{$\mathcal{U}(\mathfrak{h})$-modules}
Assume that $V=\C[\partial^2]\oplus \partial\C[\partial^2]$. Then $V$ is a $\Z_2$-graded vector space with $V_{\bar 0}=\C[\partial^2]$ and
$V_{\bar 1}=\partial\C[\partial^2]$. Now we  give an precise construction of an $\T$-module structure
on $V$.
\begin{prop}\label{pro2.21}
For $\lambda\in \C^*,\alpha\in\C,t=\pm1,f(\partial^2)\in\C[\partial^2]$ and $\partial f(\partial^2)\in \partial\C[\partial^2]$,
we define the following  action
of  twisted  $N=2$ superconformal algebra  on $V$ as follows
\begin{eqnarray}\label{2.22}
&&L_mf(\partial^2)=\lambda^m(\partial^2+m\alpha)f(\partial^2+m),
\\&&\label{2.33} L_{m}\partial f(\partial^2)=\lambda^m(\partial^2+m(\alpha+\frac{1}{2}))\partial f(\partial^2+m),
\\&& \label{2.44} I_rf(\partial^2)=-t^{2r}2\lambda^{r}\alpha f(\partial^2+r),
\\&& \label{2.55}  I_{r}\partial f(\partial^2)=t^{2r}\lambda^{r}(1-2\alpha)\partial f(\partial^2+r),
\\&& \label{2.66} G_pf(\partial^2)=t^{2p}\lambda^p\partial f(\partial^2+p),
\\&& \label{2.77} G_p\partial f(\partial^2)=(-t)^{2p}\lambda^p(\partial^2+2p\alpha)f(\partial^2+p)
\end{eqnarray}
for $m\in\Z,r\in  \frac{1}{2}+\Z, p\in\frac{1}{2}\Z.$
 Then $V$ is an  $\T$-module under the action of \eqref{2.22}-\eqref{2.77}, which
  is a
free of rank $1$ as a $\mathcal{U}(\mathfrak{h})$-module  and   denoted by $\mathcal{M}_t(\lambda,\alpha)$.
\end{prop}
\begin{proof}
According to Theorem \ref{le2.1} (1),  one can see that
\begin{eqnarray*}
&&L_mL_nf(\partial^2)-L_nL_mf(\partial^2)=(m-n)L_{m+n}f(\partial^2),
\\&&L_mL_n\partial f(\partial^2)-L_nL_m\partial f(\partial^2)=(m-n)L_{m+n}\partial f(\partial^2),
\end{eqnarray*}
where $m,n \in\Z$.

For any $m\in\Z,r\in  \frac{1}{2}+\Z$,  from \eqref{2.22}-\eqref{2.55},   we respectively check
\begin{eqnarray*}
&&L_mI_rf(\partial^2)-I_rL_mf(\partial^2)
\\&=&L_m\big(-2t^{2r}\lambda^{r}\alpha f(\partial^2+r)\big)-I_r\big(\lambda^m(\partial^2+m\alpha)f(\partial^2+m)\big)
\\&=&-2t^{2r}\lambda^{m+r}\alpha (\partial^2+m\alpha)f(\partial^2+m+r)+2t^{2r}\lambda^{m+r}\alpha (\partial^2+r+m\alpha)f(\partial^2+m+r)
\\&=&2t^{2(m+r)}\lambda^{m+r}\alpha rf(\partial^2+m+r)=-rI_{m+r}f(\partial^2)
\end{eqnarray*}
and
\begin{eqnarray*}
&&L_mI_r\partial f(\partial^2)-I_rL_m\partial f(\partial^2)
\\&=&L_m\big(t^{2r}\lambda^{r}\partial (1-2\alpha) f(\partial^2+r)\big)-I_r\big(\lambda^m\partial\big(\partial+m(\alpha+\frac{1}{2})\big) f(\partial^2+m)\big)
\\&=&t^{2r}\lambda^{m+r}\partial(1-2\alpha ) \big(\partial^2+m(\alpha+\frac{1}{2})\big)f(\partial^2+m+r)
\\&&-t^{2r}\lambda^{m+r}\partial(1-2\alpha ) \big(\partial^2+r+m(\alpha+\frac{1}{2})\big)f(\partial^2+m+r)
\\&=&-t^{2(m+r)}\lambda^{m+r}(1-2\alpha) r\partial f(\partial^2+m+r)=-rI_{m+r}\partial f(\partial^2).
\end{eqnarray*}
For any $r,s\in  \frac{1}{2}+\Z$,  by \eqref{2.44} and \eqref{2.55},  it is straightforward to verify that
\begin{eqnarray*}
&&I_rI_sf(\partial^2)-I_sI_rf(\partial^2)
\\&=&I_r(-2t^{2s}\lambda^{s}\alpha f(\partial^2+s))-I_s(-2t^{2r}\lambda^{r}\alpha f(\partial^2+r))=0,
\\&&I_rI_s\partial f(\partial^2)-I_sI_r\partial f(\partial^2)
\\&=&
I_r(t^{2s}\lambda^{s}\partial(1-2\alpha) f(\partial^2+s))-I_s(t^{2r}\lambda^{r}\partial(1-2\alpha) f(\partial^2+r))=0.
\end{eqnarray*}
Fixing $r\in\frac{1}{2}+\Z$, then $2r$ is an odd number.
For any $r\in  \frac{1}{2}+\Z,p\in\frac{1}{2}\Z$,  using \eqref{2.44}-\eqref{2.77},   we obtain
\begin{eqnarray*}
&&I_rG_pf(\partial^2)-G_pI_rf(\partial^2)
\\&=&I_r(t^{2p}\lambda^p\partial f(\partial^2+p))-G_p(-2t^{2r}\lambda^{r}\alpha f(\partial^2+r))
\\&=&t^{2(r+p)}\lambda^{r+p}\partial(1-2\alpha) f(\partial^2+r+p)+2\alpha  t^{2(r+p)}\lambda^{r+p}\partial f(\partial^2+r+p)
\\&=&t^{2(r+p)}\lambda^{r+p}\partial f(\partial^2+r+p)=G_{r+p}f(\partial^2)
\end{eqnarray*}
and
\begin{eqnarray*}
&&I_rG_p\partial f(\partial^2)-G_pI_r\partial f(\partial^2)
\\&=&I_r\big((-t)^{2p}\lambda^p(\partial^2+2p\alpha)f(\partial^2+p)\big)-G_p\big(t^{2r}\lambda^{r}\partial(1-2\alpha) f(\partial^2+r)\big)
\\&=&(-1)^{2p}t^{2(r+p)}\lambda^{r+p}(-2\alpha)(\partial^2+r+2p\alpha) f(\partial^2+r+p)
\\&&-(-1)^{2p}t^{2(r+p)}\lambda^{r+p}(\partial^2+2p\alpha)(1-2\alpha) f(\partial^2+r+p)
\\&=&(-t)^{2(p+r)}\lambda^{r+p}(\partial^2+2(p+r)\alpha) f(\partial^2+r+p)=G_{r+p}\partial f(\partial^2).
\end{eqnarray*}
Moreover, for $p,q\in\frac{1}{2}\Z$, it follows from  \eqref{2.22}-\eqref{2.77} that we have
\begin{eqnarray*}
&&G_pG_q f(\partial^2)+G_qG_p f(\partial^2)
\\&=&G_p(\lambda^q\partial f(\partial^2+q))+G_q(\lambda^p\partial f(\partial^2+p))
\\&=&(-1)^{2p}\Big(t^{2(p+q)}\lambda^{p+q}(\partial^2+2p\alpha) f(\partial^2+p+q)
\\&&+(-t)^{2(p+q)}\lambda^{p+q}(\partial^2+2q\alpha) f(\partial^2+p+q)\Big)
\\&=& \left\{\begin{array}{llll}
(-1)^{2p}2 \lambda^{p+q}\big(\partial^2+(p+q)\alpha\big) f(\partial^2+p+q)&\mbox{if}\
p+q\in\Z,\\[4pt]
(-1)^{2p+1}t^{2(p+q)}\lambda^{p+q}(p-q)(-2\alpha) f(\partial^2+p+q)&\mbox{if}\
p+q\in\frac{1}{2}+\Z,
\end{array}\right.
\\&=&
\left\{\begin{array}{llll}
(-1)^{2p}2L_{p+q}f(\partial^2)&\mbox{if}\
p+q\in\Z,\\[4pt]
(-1)^{2p+1}(p-q)I_{p+q}f(\partial^2)&\mbox{if}\
p+q\in\frac{1}{2}+\Z,
\end{array}\right.
\end{eqnarray*}
and
\begin{eqnarray*}
&&G_pG_q \partial f(\partial^2)+G_qG_p \partial f(\partial^2)
\\&=&G_p\big((-t)^{2q}\lambda^q(\partial^2+2q\alpha)f(\partial^2+q)\big)
+G_q\big((-t)^{2p}\lambda^p(\partial^2+2p\alpha)f(\partial^2+p)\big)
\\&=&(-1)^{2p}\Big((-t)^{2(p+q)}\lambda^{p+q}\partial(\partial^2+p+2q\alpha)f(\partial^2+p+q)
\\&&+t^{2(p+q)}\lambda^{p+q}\partial(\partial^2+q+2p\alpha)f(\partial^2+p+q)\Big)
\\&=& \left\{\begin{array}{llll}
(-1)^{2p}2\Big( \lambda^{p+q}\partial\big(\partial^2+(p+q)(\alpha+\frac{1}{2})\big) f(\partial^2+p+q)\Big)&\mbox{if}\
p+q\in\Z,\\[4pt]
(-1)^{2p+1}t^{2(p+q)}\lambda^{p+q}(p-q)\partial(1-2\alpha)  f(\partial^2+p+q)&\mbox{if}\
p+q\in\frac{1}{2}+\Z,
\end{array}\right.
\\&=&
\left\{\begin{array}{llll}
(-1)^{2p}\big(2L_{p+q}\partial f(\partial^2)\big)&\mbox{if}\
p+q\in\Z,\\[4pt]
(-1)^{2p+1}(p-q)I_{p+q}\partial f(\partial^2)&\mbox{if}\
p+q\in\frac{1}{2}+\Z.
\end{array}\right.
\end{eqnarray*}

Finally, for any $m\in\Z,p\in\frac{1}{2}\Z$, by the similar computation, we obtain
\begin{eqnarray*}
&&L_mG_pf(\partial^2)-G_pL_mf(\partial^2)
\\&=&L_m(t^{2p}\lambda^p\partial f(\partial^2+p))-G_p(\lambda^m(\partial^2+m\alpha)f(\partial^2+m))
\\&=&(\frac{m}{2}-p)t^{2(m+p)}\lambda^{m+p}\partial f(\partial^2+m+p))=(\frac{m}{2}-p)G_{m+p}f(\partial^2)
\end{eqnarray*}
and
\begin{eqnarray*}
&&L_mG_p\partial f(\partial^2)-G_pL_m\partial f(\partial^2)
\\&=&L_m\big((-t)^{2p}\lambda^p(\partial^2+2p\alpha)f(\partial^2+p)\big)-G_p\big(\lambda^m\partial(\partial^2+m(\alpha+\frac{1}{2}))f(\partial^2+m)\big)
\\&=&(-t)^{2p}\lambda^{m+p}(\partial^2+m\alpha)(\partial^2+m+2p\alpha)f(\partial^2+m+p))
\\&&-(-t)^{2p}\lambda^{m+p}(\partial^2+2p\alpha)(\partial^2+p+m(\alpha+\frac{1}{2}))f(\partial^2+m+p)
\\&=&(\frac{m}{2}-p)(-t)^{2p}\lambda^{m+p}(\partial^2+2\alpha(p+m)) f(\partial^2+m+p)
\\&=&(\frac{m}{2}-p)G_{m+p}\partial f(\partial^2).
\end{eqnarray*}
 This completes the proof.
\end{proof}
\begin{rema}
By the  definition of $\T$-module $\mathcal{M}_t(\lambda,\alpha)$ for $t=\pm1$, we can define a class of Witt modules
$\Omega_b(\lambda,\mu,\alpha)=\C[x]$ with the following action:
$$\mathbb{L}_{m}f(x)=\lambda^m(x+m\alpha)f(x+m)-\mu t^{2m-1}\lambda^{m-\frac{1}{2}} f(x+m-\frac{1}{2}),$$
where $m\in\Z,\lambda,\mu\in\C^*,\alpha\in\C$.
\end{rema}

The   simplicities and isomorphism classes of $\T$-modules $\mathcal{M}_t(\lambda,\alpha)$  are presented separately  in the following two theorems.
\begin{theo}\label{th2.3}
Let $\lambda\in\C^*,\alpha\in\C,t=\pm1$, $\Upsilon_t=\partial^2(\mathcal{M}_t(\lambda,0)_{\bar 0})\oplus\mathcal{M}_t(\lambda,0)_{\bar 1}$.
Then   \begin{itemize}
\item[\rm(1)]  $\mathcal{M}_t(\lambda,\alpha)$ is simple if and only if  $\alpha\neq0$. Furthermore, $\mathcal{M}_t(\lambda,0)$ has a unique  proper submodule $\Upsilon_t$, and $\mathcal{M}_t(\lambda,0)/\Upsilon_t$  is a one-dimensional trivial
$\T$-module;
\item[\rm(2)]  $\Upsilon_1\cong\Pi(\mathcal{M}_{-1}(\lambda,\frac{1}{2}))$ and  $\Upsilon_{-1}\cong\Pi(\mathcal{M}_1(\lambda,\frac{1}{2}))$  are simple $\T$-modules.
\end{itemize}
  \end{theo}
\begin{proof}
(1)
Assume that $P=P_{\bar0}\oplus P_{\bar1}$ is a nonzero submodule of $\M_t(\lambda,\alpha)$. By  \eqref{2.66} and  \eqref{2.77},
we get that   $P_{\bar0}$ and $P_{\bar1}$ are both nonzero.
If we prove $P_{\bar0}=\M_t(\lambda,\alpha)_{\bar0}$,  the simplicities of
 $\M_t(\lambda,\alpha)$ will be determined.

We have the following two cases.
\begin{case}
$\alpha=0.$
\end{case}
 It is easy to check that  $\Upsilon_t$ is a proper submodule
of $\M_t(\lambda,\alpha)$ for $\lambda\in\C^*,\alpha\in\C$.
According to  Theorem \ref{le2.1} (3), we see that   $P_{\bar0}=\partial^2(\mathcal{M}_t(\lambda,\alpha)_{\bar 0})$. Based on   \eqref{2.66},
we observe  $P_{\bar1}=\mathcal{M}_t(\lambda,\alpha)_{\bar 1}$.
Therefore,   $P=\Upsilon_t$ and  $\mathcal{M}_t(\lambda,0)/\Upsilon_t$ is a $1$-dimensional trivial $\T$-module.

\begin{case}
$\alpha\neq0.$
\end{case}
The simple Witt modules can be extended to the simple $\T_{\bar 0}$-modules.
Then from Theorem \ref{le2.1} (2),   we can conclude  that $\mathcal{M}_t(\lambda,\alpha)$
is a simple $\T$-module.

(2)
Define the following linear   map
 \begin{eqnarray*}
\chi:\quad\Upsilon_1\quad&\longrightarrow& \Pi(\mathcal{M}_{-1}(\lambda,\frac{1}{2}))
\\ \partial^2f(\partial^2)&\longmapsto& \partial f(\partial^2)
\\ \partial f(\partial^2)&\longmapsto& f(\partial^2).
\end{eqnarray*}
For $r\in\Z+\frac{1}{2},p\in\frac{1}{2}\Z$, we check that
\begin{eqnarray*}
&&\chi(G_p\partial f(\partial^2))=(-1)^{2p}\lambda^p \partial f(\partial^2+p)=G_p\chi(\partial f(\partial^2)),
\\&&\chi(G_p\partial^2f(\partial^2))=\lambda^p(\partial^2+p)f(\partial^2+p)=G_p\chi(\partial^2f(\partial^2)),
\\&&\chi(I_r\partial f(\partial^2))=\lambda^r  f(\partial^2+r)=I_r\chi(\partial f(\partial^2)),
\\&&\chi(I_r\partial^2f(\partial^2))=I_r\chi(\partial^2f(\partial^2))=0.
\end{eqnarray*}
By an identical process, one can confirm that $$
\chi(L_m\partial f(\partial^2))=L_m\chi(\partial f(\partial^2)),
\chi(L_m\partial^2f(\partial^2))=L_m\chi(\partial^2f(\partial^2)).
$$
Hence, $\chi$ is an $\T$-module isomorphism, and $\Upsilon_1$ is a simple submodule of $\T$. Using the similar arguments, we know that $\Upsilon_{-1}\cong\Pi(\mathcal{M}_1(\lambda,\frac{1}{2}))$  is also a simple $\T$-module.
\end{proof}
\begin{rema}
Using Theorem \ref{th2.3}, we have the   short exact sequence
\begin{eqnarray*}
0\rightarrow\Upsilon_t
\stackrel{\xi_1}{\longrightarrow}
\mathcal{M}_t(\lambda,0)\stackrel{\xi_2}{\longrightarrow}\C\rightarrow0,
\end{eqnarray*}
where $\xi_1$ and $\xi_2$  are respectively the embedding mapping  and canonical mapping  defined by
$$\xi_1:f(\partial^2)+\partial g(\partial^2)\rightarrow \partial^2f(\partial^2)+\partial g(\partial^2),\ \xi_2:f(\partial^2)+\partial g(\partial^2)\rightarrow \overline{f(\partial^2)+\partial g(\partial^2)}.$$
\end{rema}

%\begin{prop}
%For $\lambda,\mu\in\C^*,\alpha,\beta\in\C$ with $\lambda\neq\mu$ and $\alpha\neq\beta$, we get $\mathcal{M}(\lambda,\alpha)\ncong\mathcal{M}(\mu,\beta)$.
%\end{prop}
%\begin{proof}
%Let

%\end{proof}
\begin{theo}\label{th2ere.4}
 Let $\lambda,\mu\in\C^*,\alpha,\beta\in\C,t,t^\prime=\pm1$.
We obtain the following statements.
 \begin{itemize}
\item[\rm(i)] $\Pi\big(\mathcal{M}_t(\lambda,\alpha)\big)\ncong\mathcal{M}_{t^\prime}(\mu^\prime,\beta^\prime)$ for any $\mu^\prime,\beta^\prime\in\C$;

\item[\rm(ii)] $\mathcal{M}_t(\lambda,\alpha)\cong\mathcal{M}_{t^\prime}(\mu,\beta)$ if and only if $\lambda=\mu$, $\alpha=\beta$ and $t=t^\prime$.
\end{itemize}
\end{theo}
\begin{proof}
(i) Let
$$
\psi:\Pi\big(\mathcal{M}_t(\lambda,\alpha)\big)\longrightarrow\mathcal{M}_{t^\prime}(\mu^\prime,\beta^\prime)
$$ be an isomorphism. Assume that  $\mathbf{1}$, $\mathbf{1}^\prime$   are respectively the
generators of the free $\C[L_0,G_0]$-modules
$\mathcal{M}_t(\lambda,\alpha)$ and $\mathcal{M}_{t^\prime}(\mu^\prime,\beta^\prime).$
In fact, there exists some $e_1,e_2\in\C^*$ such that $\psi(\mathbf{1})=e_1\partial\mathbf{1}^\prime$
and $\psi(\partial\mathbf{1})=e_2\mathbf{1}^\prime$.
For $p\in\frac{1}{2}\Z$, we confirm that $$G_p\psi(\mathbf{1})=G_p\big(e_1\partial\mathbf{1}^\prime\big)
=(-t^\prime)^{2p}e_1(\mu^\prime)^p(\partial^2+2p\beta^\prime)\mathbf{1}^\prime
\ \mathrm{and} \  \psi(G_p\mathbf{1})=e_2t^{2p}\lambda^p\mathbf{1}^\prime.$$  Thus we have $G_p\psi(\mathbf{1})\neq\psi(G_p\mathbf{1})$,  which yields a contradiction.

(ii) The ``if" part is clear.
Now we consider the ``only if" part.
Suppose $\mathcal{M}_t(\lambda,\alpha)\cong\mathcal{M}_{t^\prime}(\mu,\beta)$  as $\T$-modules. It follows from (i)
that $\mathcal{M}_t(\lambda,\alpha)_{\bar 0}\cong\mathcal{M}_{t^\prime}(\mu,\beta)_{\bar 0}$ can be viewed as   $\mathcal{V}$-modules.  By Theorem 12 of \cite{LZ},  one has  $\lambda=\mu$ and $\alpha=\beta$.

Let
$$
\varphi:\mathcal{M}_t(\lambda,\alpha)\longrightarrow\mathcal{M}_{t^\prime}(\mu,\beta)
$$ be an isomorphism, and let  $\mathbf{1}$, $\mathbf{1}^\prime$   be the
generators of
$\mathcal{M}_t(\lambda,\alpha)$ and $\mathcal{M}_{t^\prime}(\mu,\beta).$
Then there exists some $d\in\C^*$ such that $\varphi(\mathbf{1})=d\mathbf{1}^\prime$ and $\varphi(\partial\mathbf{1})=d\partial\mathbf{1}^\prime$. For $p\in\frac{1}{2}\Z$, from $\varphi(G_p\mathbf{1})=G_p\varphi(\mathbf{1})$,
we deduce that $t=t^\prime$.
\end{proof}

\subsection{$\mathcal{U}(\mathfrak{t})$-modules}

Let $W=\C[x]\oplus\C[y]$ be an $\T$-module such that it is free of rank $2$ as a $\mathcal{U}(\C L_0)$-module
with two homogeneous basis elements $1_{\bar0}$ and $1_{\bar1}$. Then $W$ is a $\Z_2$-graded vector space with $W_{\bar 0}=\C[x]\mathbf{1}_{\bar0}$ and $W_{\bar 1}=\C[y]\mathbf{1}_{\bar1}$.
The  $\T$-module structure on $W$   denoted by    $\mathcal{N}_t(\lambda,\alpha)$   can be characterized as follows.
\begin{prop}
For $\lambda\in \C^*,\alpha\in\C,t=\pm1,f(x)\in\C[x]$ and $g(y)\in\C[y]$,
 we have the action
of   $\T$ on $W$  defined as
\begin{eqnarray}\label{4.22}
&&L_mf(x)=\lambda^m(x+m\alpha)f(x+m),
\\&&\label{4.33} L_{m}g(y)=\lambda^m\big(y+m(\alpha+\frac{1}{2})\big)g(y+m),
\\&& \label{4.44} I_rf(x)=-2t^{2r}\lambda^{r}\alpha f(x+r),
\\&& \label{4.55}  I_{r}g(y)=t^{2r}\lambda^{r}(1-2\alpha) g(y+r),
\\&& \label{4.66} G_pf(x)=t^{2p}\lambda^pf(y+p),
\\&& \label{4.77} G_pg(y)=(-t)^{2p}\lambda^p(x+2p\alpha)g(x+p),
\end{eqnarray}
where $m\in\Z,r\in\frac{1}{2}+\Z,p\in\frac{1}{2}\Z$.
 Then $W$ is an  $\T$-module under the action of \eqref{4.22}-\eqref{4.77}, which
  is a
free of rank $2$ as a module over $\mathcal{U}(\C L_0)$.
\end{prop}
\begin{proof}
Since the proof is similar to the Proposition \ref{pro2.21}, we omit the details.
\end{proof}
 %In the following, we give a class of $\T$-module  $\mathcal{N}_2(\lambda,\alpha)$, which is  a little different from $\mathcal{N}_1(\lambda,\alpha)$.
%\begin{prop}
%For $\lambda\in \C^*,\alpha\in\C,f(x)\in\C[x]$ and $g(y)\in\C[y]$,
 %we get the action
%of   $\T$ on $W$   as
%\begin{eqnarray*}
%&&L_mf(x)=\lambda^m(x+m\alpha)f(x+m),
%\\&& L_{m}g(y)=\lambda^m\big(y+m(\alpha+\frac{1}{2})\big)g(y+m),
%\\&&  I_rf(x)=2\lambda^{r}\alpha f(x+r),
%\\&&   I_{r}g(y)=-\lambda^{r}(1-2\alpha) g(y+r),
%\\&&   G_pf(x)=(-1)^{2p}\lambda^pf(y+p),
%\\&& \label{lmf2}   G_pg(y)=\lambda^p(x+2p\alpha)g(x+p),
%\end{eqnarray*}
%where $m\in\Z,r\in\frac{1}{2}+\Z,p\in\frac{1}{2}\Z$.
 %Then $W$ is an  $\T$-module under the above action, which
 % is a
%free of rank $2$ as a $\mathcal{U}(\C L_0)$-module.
%\end{prop}

\begin{lemm}\label{lemm2.4}
 Let $\lambda\in\C^*,\alpha\in\C,t=\pm1$.
Then  $\mathcal{M}_t(\lambda,\alpha)\cong\mathcal{N}_t(\lambda,\alpha)$.
\end{lemm}

\begin{proof}
 Let
\begin{eqnarray*}
\Theta:\mathcal{M}_t(\lambda,\alpha)&\longrightarrow&\mathcal{N}_t(\lambda,\alpha)
\\f(\partial^2)&\longmapsto&f(x)
\\ \partial g(\partial^2)&\longmapsto&g(y)
\end{eqnarray*}
be the linear map.
 Clearly, $\Theta$ is bijective.   Now we  are ready to  prove that it is an $\T$-module isomorphism.

For any $m\in\Z, r\in\frac{1}{2}+\Z, p\in\frac{1}{2}\Z,f(\partial^2)\in\C[\partial^2],\partial g(\partial^2)\in\partial\C[\partial^2],f(x)\in\C[x],g(y)\in\C[y]$, we have
\begin{eqnarray*}
\Theta(L_mf(\partial^2))
&=&\Theta\Big(\lambda^m(\partial^2+m\alpha)f(\partial^2+m)\Big)
\\
&=&\lambda^m(x+m\alpha)f(x+m)=L_m\Theta(f(\partial^2)),
\\ \Theta(L_m\partial g(\partial^2))
&=&\Theta\Big(\lambda^m\partial(\partial^2+m(\alpha+\frac{1}{2})) g(\partial^2+m)\Big)
\\
&=&\lambda^m(y+m(\alpha+\frac{1}{2}))g(y+m)=L_m\Theta(\partial f(\partial^2)),
\\ \Theta(I_rf(\partial^2))
&=&\Theta\Big(-2t^{2r}\lambda^{r}\alpha f(\partial^2+r)\Big)
\\
&=&-2t^{2r}\lambda^{r}\alpha f(x+r)=I_r\Theta(f(\partial^2)),
\end{eqnarray*}
\begin{eqnarray*}
\Theta(I_r\partial g(\partial^2))
&=&\Theta\Big(t^{2r}\lambda^{r}(1-2\alpha) \partial g(\partial^2+r)\Big)
\\
&=&t^{2r}\lambda^{r}(1-2\alpha) g(y+r)=I_r\Theta(\partial g(\partial^2)),
\\ \Theta(G_pf(\partial^2))
&=&\Theta\Big(t^{2p}\lambda^p f(x+p)\Big)
\\
&=&t^{2p}\lambda^p f(x+p)=G_p\Theta(f(\partial^2)),
\\ \Theta(G_p\partial g(\partial^2))
&=&\Theta\Big((-t)^{2p}\lambda^p(\partial^2+2p\alpha)g(\partial^2+p)\Big)
\\
&=&(-t)^{2p}\lambda^p(x+2p\alpha)g(x+p)=G_p\Theta(\partial g(\partial^2)).
\end{eqnarray*}
Thus,  $\Theta$   is an $\T$-module isomorphism.
\end{proof}
As a direct consequence of   Theorems \ref{th2.3}, \ref{th2ere.4} and Lemma \ref{lemm2.4},  we
get the following two theorems.
\begin{theo}\label{th346782.3}
 Assume that  $\lambda,\mu\in\C^*,\alpha,\beta\in\C,t=\pm1$ and $\pi_t=x(\mathcal{N}_t(\lambda,0)_{\bar 0})\oplus\mathcal{N}_t(\lambda,0)_{\bar 1}$.
Then
\begin{itemize}
\item[\rm(1)]   $\mathcal{N}_t(\lambda,\alpha)$ is simple if and only if  $\alpha\neq0$. In addition,  $\mathcal{N}_t(\lambda,0)$ has a unique   proper submodule $\pi_t$, and $\mathcal{N}_t(\lambda,0)/\pi_t\cong\C$.
\item[\rm(2)]
$\pi_1\cong\Pi(\mathcal{N}_{-1}(\lambda,\frac{1}{2}))$ and  $\pi_{-1}\cong\Pi(\mathcal{N}_1(\lambda,\frac{1}{2}))$  are simple $\T$-modules.
\end{itemize}
  \end{theo}
\begin{theo}\label{th9082.4}
 Let $\lambda,\mu\in\C^*,\alpha,\beta\in\C,t,t^\prime=\pm1$.
We have
\begin{itemize}
\item[\rm(i)]    $\Pi\big(\mathcal{N}_t(\lambda,\alpha)\big)\ncong\mathcal{N}_{t^\prime}(\mu^\prime,\beta^\prime)$ for any $\mu^\prime,\beta^\prime\in\C$;

\item[\rm(ii)]     $\mathcal{N}_t(\lambda,\alpha)\cong\mathcal{N}_{t^\prime}(\mu,\beta)$ if and only if $\lambda=\mu,\alpha=\beta$ and $t=t^\prime$.
\end{itemize}
  \end{theo}
\section{ Classification of free $\mathcal{U}(\mathfrak{h})$-modules of rank 1}
Assume that  $V=V_{\bar 0}\oplus V_{\bar1}$ is an $\T$-module such that it is free of rank $1$ as a $\mathcal{U}(\mathfrak{h})$-module,
where $\mathfrak{h}=\C L_0\oplus \C G_0$. It follows from  the superalgebra
structure of $\T$ in \eqref{def1.1} that  we have
$$L_0G_0=G_0L_0 \ \mathrm{and}\ G_0^2=L_0.$$
 Thus $\mathcal{U}(\mathfrak{h})=\C[L_0]\oplus G_0\C[L_0]$. Choose a homogeneous basis element $\mathbf{1}$ in $V$. Without loss of generality, up
to a parity, we may assume $\mathbf{1}\in V_{\bar 0}$  and
$$V=\mathcal{U}(\mathfrak{h})\mathbf{1}=\C[L_0]\mathbf{1}\oplus G_0\C[L_0]\mathbf{1}$$
with $V_{\bar 0}=\C[L_0]\mathbf{1}$ and $V_{\bar1}=G_0\C[L_0]\mathbf{1}$.
Then  we can  suppose that $G_0\mathbf{1}=\partial\mathbf{1},L_0\mathbf{1}=\partial^2 \mathbf{1}$.
In the following,  $V=V_{\bar 0}\oplus V_{\bar1}$ is equal to
$\C[\partial^2]\mathbf{1}\oplus \partial\C[\partial^2]\mathbf{1}$ with $V_{\bar0}=\C[\partial^2]\mathbf{1}$ and $V_{\bar1}=\partial\C[\partial^2]\mathbf{1}$.

From Theorem \ref{le2.1} (4), for
any $m\in\Z,f(\partial^2)\in \C[\partial^2]$,  there exists $\lambda\in\C^*,\alpha\in\C$ such that
\begin{eqnarray}\label{lmf31}
&&L_m(f(\partial^2))=\lambda^m(\partial^2+m\alpha)f(\partial^2+m).
\end{eqnarray}
It is clear  that $V_{\bar 0}$  can be regarded as an $\mathcal{L}$-module which is free of rank $1$ as a   $\C[L_0]$-module.
 \begin{lemm}\label{lem3.2}
For $r\in\frac{1}{2}+\Z,p\in\frac{1}{2}\Z, f(\partial^2)\in\C[\partial^2],\partial f(\partial^2)\in \partial\C[\partial^2]$, we get
\begin{itemize}
\item[\rm(1)] $I_r\partial f(\partial^2)\mathbf{1}=f(\partial^2+r)I_r\partial\mathbf{1}$;

\item[\rm(2)] $I_rf(\partial^2)\mathbf{1}=f(\partial^2+r)I_r\mathbf{1}$;

\item[\rm(3)] $G_p\partial f(\partial^2)\mathbf{1}=f(\partial^2+p)G_p\partial\mathbf{1}$;

\item[\rm(4)] $G_pf(\partial^2)\mathbf{1}=f(\partial^2+p)G_p\mathbf{1}$.
\end{itemize}
\end{lemm}
\begin{proof}  (1)  It is easy to get that
$I_rL_0\partial\mathbf{1}=(L_0+r)I_r\partial\mathbf{1}$ by  the
relations of $\T$ in \eqref{def1.1}.
Recursively, we conclude  that
$I_rL_0^n\partial\mathbf{1}=(L_0+r)^nI_r\partial\mathbf{1}$ for $n\in\Z_+$.
Hence,
$I_r\partial f(\partial^2)\mathbf{1}=f(\partial^2+r)I_r\partial\mathbf{1}$ for $r\in\frac{1}{2}+\Z$.

Similarly, we obtain (2), (3) and (4).
\end{proof}

 \begin{lemm}\label{le3.33}
For $r\in\frac{1}{2}+\Z,p\in\frac{1}{2}\Z,\lambda\in\C^*,\alpha\in\C$, one of the following two cases will take place.
\begin{itemize}
\item[\rm(a)] $G_p\mathbf{1}=\lambda^p\partial\mathbf{1},  G_r\partial\mathbf{1}=-\lambda^{r}(\partial^2+2r\alpha)\mathbf{1},  I_r\mathbf{1}=-2\lambda^{r}\alpha\mathbf{1}$;

\item[\rm(b)] $G_p\mathbf{1}=(-1)^{2p}\lambda^p\partial\mathbf{1},  G_r\partial\mathbf{1}=\lambda^{r}(\partial^2+2r\alpha)\mathbf{1},  I_r\mathbf{1}=2\lambda^{r}\alpha\mathbf{1}$.
\end{itemize}
\end{lemm}
\begin{proof} To prove this, we    suppose
\begin{eqnarray*}
G_p\mathbf{1} = \partial f_p(\partial^2)\mathbf{1}\in \partial\C[\partial^2]\mathbf{1},  I_r\mathbf{1} = g_r(\partial^2)\mathbf{1}\in \C[\partial^2]\mathbf{1},
G_p\partial\mathbf{1} = h_p(\partial^2)\mathbf{1}\in \C[\partial^2]\mathbf{1},
\end{eqnarray*} where $p\in\frac{1}{2}\Z,r\in\frac{1}{2}+\Z$.
For any $m\in\Z$, by $[G_0,G_m]\mathbf{1}=2L_m\mathbf{1}$, we get
\begin{eqnarray*}
G_m\partial\mathbf{1}=\big(2\lambda^m(\partial^2+m\alpha)-\partial^2f_m(\partial^2)\big)\mathbf{1}.
\end{eqnarray*}
Using $G_m^2\mathbf{1}=L_{2m}\mathbf{1}$, one can check that
\begin{eqnarray}\label{12345.222}
f_m(\partial^2+m)(2\lambda^m(\partial^2+m\alpha)-\partial^2f_m(\partial^2))
=\lambda^{2m}(\partial^2+2m\alpha).
\end{eqnarray}
Comparing the degree of  $\partial^2$   in \eqref{12345.222}, one has
\begin{eqnarray}\label{3.222}
f_m(\partial^2)=a_m\in\C
\end{eqnarray}for $m\in\Z$.
Taking \eqref{3.222} into \eqref{12345.222}, it is easy to get
\begin{eqnarray}\label{am3.222}
a_m=\lambda^m.
\end{eqnarray}

For any   $r\in \frac{1}{2}+\Z$, from  $[G_{r},G_{r}]\mathbf{1}=-2L_{2r}\mathbf{1}$,
we   have
\begin{eqnarray}\label{fr35}
f_{r}(\partial^2+r)h_{r}(\partial^2)
=-\lambda^{2r}(\partial^2+2r\alpha).
\end{eqnarray}
This implies that either $f_{r}(\partial^2+r)$ or $h_{r}(\partial^2)$ is a nonzero constant.
Suppose  $h_{r}(\partial^2)=d_r\in\C^*$. For $m\in\Z,r\in\Z+\frac{1}{2}$, then it follows from  $[L_{m},G_{r}]\partial\mathbf{1}=(\frac{m}{2}-r)G_{m+r}\partial\mathbf{1}$
that we get
\begin{eqnarray}\label{334423}
(r+\frac{m}{2})\lambda^md_r=(r-\frac{m}{2})d_{m+r}.
\end{eqnarray}
Taking $m=2r$ in \eqref{334423}, one has $d_r=0$, which  leads to conflict.
Then considering  \eqref{fr35} again,  we conclude
\begin{eqnarray}\label{fr334423}
f_{r}(\partial^2+r)=c_r\in\C^*, h_{r}(\partial^2)=-\frac{1}{c_r}\lambda^{2r}(\partial^2+2r\alpha).
\end{eqnarray}
%For any   $r\in \frac{1}{2}+\Z$, from $[I_{r},G_{r}]\mathbf{1}=G_{2r}\mathbf{1}$ and $[G_{r},G_{r}]\mathbf{1}=-2L_{2r}\mathbf{1}$,
%we respectively  have
%\begin{eqnarray}
%&&f_{r}(\partial^2+r)g_{r}(\partial^2)-g_{r}(\partial^2+r)f_{r}(\partial^2)
%+f_{r}(\partial^2+r)f_{r}(\partial^2)=\lambda^{2r},
%\\&&f_{r}(\partial^2+r)\big(\partial^2f_{r}(\partial^2)-rg_{r}(\partial^2)\big)
%=\lambda^{2r}(\partial^2+2r\alpha).
%\end{eqnarray}
%which implies $f_{r}(\partial^2)=\lambda^{r}$.
For $r\in\Z+\frac{1}{2}$,   by  $[G_{0},G_{r}]\mathbf{1}=rI_{r}\mathbf{1}$,
 we have
 \begin{eqnarray}\label{35532}
 g_r(\partial^2)=\frac{1}{rc_r}((c_r^2-\lambda^{2r})\partial^2-2r\alpha\lambda^{2r}).
 \end{eqnarray}
For $r,s\in\Z+\frac{1}{2}$,  it follows from \eqref{35532}
that $$I_{r}I_{s}\mathbf{1}=\frac{1}{rsc_rc_s}((c_s^2-\lambda^{2s})(\partial^2+r)-2s\alpha\lambda^{2s})((c_r^2-\lambda^{2r})\partial^2-2r\alpha\lambda^{2r}).$$
Based on $[I_{r},I_{s}]\mathbf{1}=0$,
 we deduce that
 $$r(c_s^2-\lambda^{2s})\big((c_r^2-\lambda^{2r})\partial^2-2\lambda^{2r}r\alpha\big)
 =s(c_r^2-\lambda^{2r})\big((c_s^2-\lambda^{2s})\partial^2-2\lambda^{2s}s\alpha\big).$$
 This gives $c_r=\pm\lambda^r$  for $r\in\Z+\frac{1}{2}$.
Combining with \eqref{am3.222}, we have
 $$f_{p}(\partial^2)=\lambda^p\ \mathrm{or}\ f_{p}(\partial^2)=(-1)^{2p}\lambda^p$$ for $p\in\frac{1}{2}\Z$.
We separately put   them
into  \eqref{fr334423} and
\eqref{35532},  one can obtain
  (a) and (b).
\end{proof}

Now we present the main result of this section, which gives a
complete classification of free $\mathcal{U}(\mathfrak{h})$-modules of rank 1 over $\T$.
\begin{theo}\label{th3.5}
 Assume that $V$ is an   $\T$-module such that the restriction of $V$ as a $\mathcal{U}(\mathfrak{h})$-module is free of rank $1$.
 Then up to a
parity, $V\cong\M_1(\lambda,\alpha)$, or $V\cong\M_{-1}(\lambda,\alpha)$ for $\lambda\in\C^*,\alpha\in\C$.
\end{theo}

\begin{proof}
 We first consider  Lemma \ref{le3.33} (a).    For $r\in\Z+\frac{1}{2},p\in\frac{1}{2}\Z$, it follows from Lemma  \ref{lem3.2} that we obtain
\begin{eqnarray}
&&G_pf(\partial^2)\mathbf{1}=f(\partial^2+p)G_p\mathbf{1}=\lambda^p\partial f(\partial^2+p)\mathbf{1},\label{gpf38}
\\&&I_rf(\partial^2)\mathbf{1}=f(\partial^2+r)I_r\mathbf{1}=-2\lambda^{r}\alpha f(\partial^2+r)\mathbf{1},\label{irf39}
\\&&G_r\partial f(\partial^2)\mathbf{1}= f(\partial^2+r)G_r\partial\mathbf{1}=-\lambda^r(\partial^2+2r\alpha)f(\partial^2+r)\mathbf{1}.\label{iasrf39}
\end{eqnarray}
For $r\in\Z+\frac{1}{2},$ we can check  that
  \begin{eqnarray}\label{ir310}
  &&I_r\partial f(\partial^2)\mathbf{1}=G_0I_rf(\partial^2)\mathbf{1}+G_rf(\partial^2)\mathbf{1}=\lambda^{r}\partial(1-2\alpha) f(\partial^2+r)\mathbf{1}.
  %\\&&G_{2r}\partial f(\partial^2)\mathbf{1}=I_rG_r\partial f(\partial^2)\mathbf{1}+G_rf(\partial^2)\mathbf{1}=\lambda^{r}\partial(1-2\alpha) f(\partial^2+r)\mathbf{1}.\label{2r39}
 \end{eqnarray}

Besides, for $0\neq m\in\Z$, we confirm that
\begin{eqnarray}\nonumber
L_{m}\partial f(\partial^2)\mathbf{1}
&=&L_{m}G_0f(\partial^2)\mathbf{1}\nonumber
\\&=&G_0L_{m}f(\partial^2)\mathbf{1}+\frac{m}{2}G_{m}f(\partial^2)\mathbf{1}\nonumber
\\&=&\lambda^m\partial(\partial^2+m(\alpha+\frac{1}{2}))f(\partial^2+m)\mathbf{1} \label{lmp311}
\end{eqnarray}
and
\begin{eqnarray}
G_{m}\partial f(\partial^2)\mathbf{1}\nonumber
&=&\frac{2}{m}\big(L_{m}G_0\partial f(\partial^2)\mathbf{1}-G_0L_m\partial f(\partial^2)\mathbf{1}\big)
\\&=&\frac{2}{m}\lambda^m(\frac{m}{2}\partial^2+m^2\alpha)f(\partial^2+m)\nonumber
\\&=&\lambda^m(\partial^2+2m\alpha)f(\partial^2+m)\mathbf{1}.\label{gm311}
\end{eqnarray}
Then \eqref{iasrf39} and \eqref{gm311} give
\begin{eqnarray}\label{gp09}
G_{p}\partial f(\partial^2)\mathbf{1}=(-1)^{2p}\lambda^p(\partial^2+2p\alpha)f(\partial^2+m)\mathbf{1}
\end{eqnarray} for $p\in\frac{1}{2}\Z$.
Thus, \eqref{lmf31}, \eqref{gpf38}, \eqref{irf39}, \eqref{ir310}, \eqref{lmp311} and \eqref{gp09} show that $V\cong\M_1(\lambda,\alpha)$ as $\T$-modules.

Now consider Lemma \ref{le3.33} (b). From Lemma  \ref{lem3.2}  and by the similar discussions,  we deduce $V\cong\M_{-1}(\lambda,\alpha)$.
%\begin{eqnarray*}
%&&G_pf(\partial^2)\mathbf{1}=(-1)^{2p}\lambda^p\partial f(\partial^2+p)\mathbf{1},
%\\&&I_rf(\partial^2)\mathbf{1}=2\lambda^{r}\alpha f(\partial^2+r)\mathbf{1},
%\\&&G_p\partial f(\partial^2)\mathbf{1}=\lambda^p(\partial^2+2p\alpha)f(\partial^2+m)\mathbf{1},
%\\&&I_r\partial f(\partial^2)\mathbf{1}=-\lambda^{r}\partial(1-2\alpha) f(\partial^2+r)\mathbf{1},
%\\&&L_{m}\partial f(\partial^2)\mathbf{1}
%=\lambda^m\partial(\partial^2+m(\alpha+\frac{1}{2}))f(\partial^2+m)\mathbf{1}.
%\end{eqnarray*}
 The theorem is proved.
\end{proof}

\section{Classification of free $\mathcal{U}(\mathfrak{t})$-modules of rank $2$}
Let  $W=W_{\bar 0}\oplus W_{\bar1}$ be an $\T$-module
 such that it is free of rank $2$ as a $\mathcal{U}(\mathfrak{t})$-module
with two homogeneous basis elements $v$ and $w$. If the parities of $v$ and $w$ are the same, for
$r\in\frac{1}{2}+\Z$ then
 $G_{\pm r}v=G_{\pm r}w=0$. Thus,
$$L_{0}v=-\frac{1}{2}[G_{r},G_{-r}]v=0,\   L_{0}w=-\frac{1}{2}[G_{r},G_{-r}]w=0,$$
which  is  incompatible with $\mathcal{U}(\C L_0)$-free.
We know that $v$ and $w$ are different parities. Set $v=\mathbf{1}_{\bar 0}\in W_{\bar 0}$ and
$w=\mathbf{1}_{\bar 1}\in W_{\bar 1}$. As a vector space, we have $W_{\bar 0}= \C[x]\mathbf{1}_{\bar 0}$ and $W_{\bar 1}= \C[y]\mathbf{1}_{\bar 1}$.

Clearly,
  $W_{\bar 0}$ and $W_{\bar 1}$  are both can be viewed     as $\mathcal{V}$-modules.
  According to Theorem \ref{le2.1} (1),
there exist
 $\lambda,\mu\in \C^*,\alpha,\beta\in\C,f(x)\in\C[x]$ and $g(y)\in\C[y]$
such that
\begin{eqnarray}\label{eqqqw4.1}
&&L_mf(x)=\lambda^m(x+m\alpha)f(x+m),
\\&& L_mg(y)=\mu^m(y+m\beta)g(y+m).\label{eqqew4.2}
\end{eqnarray}

For later use,   two preliminary lemmas are presented as follows.
 \begin{lemm}\label{lem4.21452}
 Let $\lambda,\mu\in\C^*,\alpha,\beta\in\C,f(x)\in\C[x],g(y)\in\C[y]$. We obtain $\lambda=\mu$ and  one
of the following two cases occurs.
\begin{itemize}
\item[\rm(i)]  $G_{0}\mathbf{1}_{\bar0}=\mathbf{1}_{\bar1},  G_{0}\mathbf{1}_{\bar1}
    =x\mathbf{1}_{\bar0}, G_{\frac{1}{2}}\mathbf{1}_{\bar0}=\lambda^{\frac{1}{2}}\mathbf{1}_{\bar1},  G_{\frac{1}{2}}\mathbf{1}_{\bar1}
    =-\lambda^{\frac{1}{2}}(x+\alpha)\mathbf{1}_{\bar0},  \beta=\alpha+\frac{1}{2}$,
    \\ $\mathrm{or}$\  $G_{0}\mathbf{1}_{\bar0}=\mathbf{1}_{\bar1},  G_{0}\mathbf{1}_{\bar1}
    =x\mathbf{1}_{\bar0}, G_{\frac{1}{2}}\mathbf{1}_{\bar0}=-\lambda^{\frac{1}{2}}\mathbf{1}_{\bar1},  G_{\frac{1}{2}}\mathbf{1}_{\bar1}
    =\lambda^{\frac{1}{2}}(x+\alpha)\mathbf{1}_{\bar0},  \beta=\alpha+\frac{1}{2}$;
\item[\rm(ii)] $G_{0}\mathbf{1}_{\bar0}
    =y\mathbf{1}_{\bar1},  G_{0}\mathbf{1}_{\bar1}
    =\mathbf{1}_{\bar0}, G_{\frac{1}{2}}\mathbf{1}_{\bar0}
    =-\lambda^{\frac{1}{2}}(y+\alpha-\frac{1}{2})\mathbf{1}_{\bar1},  G_{\frac{1}{2}}\mathbf{1}_{\bar1}
    =\lambda^{\frac{1}{2}}\mathbf{1}_{\bar0},  \beta=\alpha-\frac{1}{2}$,
   \\ $\mathrm{or}$\  $G_{0}\mathbf{1}_{\bar0}
    =y\mathbf{1}_{\bar1},  G_{0}\mathbf{1}_{\bar1}
    =\mathbf{1}_{\bar0}, G_{\frac{1}{2}}\mathbf{1}_{\bar0}
    =\lambda^{\frac{1}{2}}(y+\alpha-\frac{1}{2})\mathbf{1}_{\bar1},  G_{\frac{1}{2}}\mathbf{1}_{\bar1}
    =-\lambda^{\frac{1}{2}}\mathbf{1}_{\bar0},  \beta=\alpha-\frac{1}{2}$.
\end{itemize}
\end{lemm}
\begin{proof} Suppose  $G_{p}\mathbf{1}_{\bar0}=f_p(y)\mathbf{1}_{\bar1}$ and $G_{p}\mathbf{1}_{\bar1}=g_p(x)\mathbf{1}_{\bar0}$.
For $p\in\{0,\frac{1}{2}\}$, using
$$G_{p}^2\mathbf{1}_{\bar0}=f_p(x+p)g_p(x)\mathbf{1}_{\bar0}\ \mathrm{and} \ L_{2p}\mathbf{1}_{\bar0}= \lambda^{2p}(x+2p\alpha)$$ in $[G_{p},G_{p}]\mathbf{1}_{\bar0}=(-1)^{2p}2L_{2p}\mathbf{1}_{\bar0}$,   we check
\begin{eqnarray}\label{gft4.1}
f_p(x+p)g_p(x)=(-1)^{2p}\lambda^{2p}(x+2p\alpha).
\end{eqnarray}
Choosing $p=0$ in \eqref{gft4.1}, we immediately obtain
\begin{eqnarray}\label{eq4.1}
&&f_{0}(x)=\gamma_{0},g_{0}(x) =\frac{1}{\gamma_{0}}x\   \mathrm{or} \ f_{0}(x)=\frac{1}{\gamma_{0}}x,g_{0}(x) =\gamma_{0}.
\end{eqnarray}
Up to a parity,   we can assume  $\gamma_0=1$ without loss of generality.
Then \eqref{eq4.1} can be rewritten as
\begin{eqnarray*}
&&f_{0}(x)=1,g_{0}(x) =x\  \mathrm{or} \  f_{0}(x)=x,g_{0}(x) =1.
\end{eqnarray*}
Choosing $p=\frac{1}{2}$ in \eqref{gft4.1}, we have
\begin{eqnarray}\label{of44}
 f_{\frac{1}{2}}(x+\frac{1}{2})=\gamma_{\frac{1}{2}},g_{\frac{1}{2}}(x) =-\frac{1}{\gamma_{\frac{1}{2}}}\lambda(x+\alpha)
  \  \mathrm{or} \    f_{\frac{1}{2}}(x+\frac{1}{2})=-\frac{1}{\gamma_{\frac{1}{2}}}\lambda(x+\alpha),g_{\frac{1}{2}}(x) =\gamma_{\frac{1}{2}}.
\end{eqnarray}
Similarly, by $G_{{\frac{1}{2}}}^2\mathbf{1}_{\bar1}=-L_{1}\mathbf{1}_{\bar1}$, we can conclude that  $g_{\frac{1}{2}}(y+{\frac{1}{2}})f_{\frac{1}{2}}(y)=-\mu  (y+\beta)$.
According to  \eqref{of44},
it is easy to  check that     $\lambda=\mu,\beta=\alpha+\frac{1}{2}$ or $\lambda=\mu,\beta=\alpha-\frac{1}{2}$.
Then we have the following four cases.
\begin{itemize}
\item[\rm(1)]  $f_{0}(x)=1,g_{0}(x) =x, f_{\frac{1}{2}}(x)=\gamma_{\frac{1}{2}},g_{\frac{1}{2}}(x) =-\frac{\lambda}{\gamma_{\frac{1}{2}}}(x+\alpha),  \beta=\alpha+\frac{1}{2}$;
 \item[\rm(2)]  $f_{0}(x)=1,g_{0}(x) =x, f_{\frac{1}{2}}(x)=-\frac{\lambda}{\gamma_{\frac{1}{2}}}(x+\alpha-\frac{1}{2}),g_{\frac{1}{2}}(x) =\gamma_{\frac{1}{2}},  \beta=\alpha-\frac{1}{2}$;
\item[\rm(3)] $f_{0}(x)=x,g_{0}(x) =1, f_{\frac{1}{2}}(x)=-\frac{\lambda}{\gamma_{\frac{1}{2}}}(x+\alpha-\frac{1}{2}),g_{\frac{1}{2}}(x) =\gamma_{\frac{1}{2}},  \beta=\alpha-\frac{1}{2}$;

\item[\rm(4)] $f_{0}(x)=x,g_{0}(x) =1,  f_{\frac{1}{2}}(x)=\gamma_{\frac{1}{2}},g_{\frac{1}{2}}(x) =-\frac{\lambda}{\gamma_{\frac{1}{2}}}(x+\alpha),  \beta=\alpha+\frac{1}{2}$.
\end{itemize}
\begin{clai}
 $(2)$ and $(4)$ do not occur.
\end{clai}
First  consider case $(2)$.
For $0\neq m\in\Z$,
we get
\begin{eqnarray}\label{23waq}
G_m\mathbf{1}_{\bar 0}&=&\frac{2}{m}(L_mG_0-G_0L_m)\mathbf{1}_{\bar 0}\nonumber
\\&=&\frac{2}{m}(L_m\mathbf{1}_{\bar 1}-G_0\lambda^m(x+m\alpha)\mathbf{1}_{\bar 0})\nonumber
\\&=&-\lambda^m\mathbf{1}_{\bar 1}.
\end{eqnarray}
Taking \eqref{23waq} into
 $([L_m,G_{-m}])\mathbf{1}_{\bar 0}=\frac{3m}{2}G_0\mathbf{1}_{\bar 0}$,
which  yields  a contradiction by $\mathbf{1}_{\bar 1}=0$.
From the similar arguments, we obtain that $(4)$ does not occur, too.
The claim holds.

Now we consider  (1).  For $r\in\frac{1}{2}+\Z$, write $I_{r}\mathbf{1}_{\bar 0}=h_{r}(x)\mathbf{1}_{\bar 0}$. By  $[G_{0},G_{\frac{1}{2}}]\mathbf{1}_{\bar 0}=\frac{1}{2}I_{\frac{1}{2}}\mathbf{1}_{\bar 0}$,
 we have
 \begin{eqnarray}\label{h1235532}
 h_{\frac{1}{2}}(x)=\frac{2}{\gamma_{\frac{1}{2}}}((\gamma_{\frac{1}{2}}^2-\lambda)x-\lambda\alpha).
 \end{eqnarray}
Then from  $[L_{1},I_{\frac{1}{2}}]\mathbf{1}_{\bar 0}=-\frac{1}{2}I_{\frac{3}{2}}\mathbf{1}_{\bar 0}$,
 we observe
 \begin{eqnarray}\label{qw35532}
 h_{\frac{3}{2}}(x)=-\frac{4\lambda}{\gamma_{\frac{1}{2}}}
 \big((\gamma_{\frac{1}{2}}^2-\lambda)(\frac{1}{2}x+\alpha)+\frac{1}{2}\lambda\alpha\big).
 \end{eqnarray}
Applying   \eqref{h1235532} and \eqref{qw35532}  to $[I_{\frac{1}{2}},I_{\frac{3}{2}}]\mathbf{1}_{\bar0}=0$,
 we see that
 $$(\gamma_{\frac{1}{2}}^2-\lambda)((\gamma_{\frac{1}{2}}^2-\lambda)x-\lambda\alpha)
 =3(\gamma_{\frac{1}{2}}^2-\lambda)((\gamma_{\frac{1}{2}}^2-\lambda)(x+2\alpha)+\lambda\alpha),$$
which  implies  $\gamma_{\frac{1}{2}}=\pm\lambda^\frac{1}{2}$. Plugging this into (1), we have (i).

By an identical process  in (3),  we know that  the results of (ii). This completes the proof.
\end{proof}

Up to a parity, we only study     (i)   of Lemma  \ref{lem4.21452}.
%A
%complete classification of free $\mathcal{U}(\mathfrak{t})$-modules of rank $2$ over $\T$ will be determined.
\begin{lemm}\label{lem4.33}
For any $m\in\Z,r \in \frac{1}{2}+\Z,p\in\frac{1}{2} \Z,f(x)\in\C[x], g(y)\in\C[y],$ we have
\begin{eqnarray*}
&&G_pf(x)\mathbf{1}_{\bar 0}=\lambda^{p}f(y+p)\mathbf{1}_{\bar 1},
\
G_pg(y)\mathbf{1}_{\bar 1}=(-1)^{2p}\lambda^{p}(x+2p\alpha)g(x+p)\mathbf{1}_{\bar 0},
\\&&
I_rf(x)\mathbf{1}_{\bar 0}=-2\alpha\lambda^{r}f(x+r)\mathbf{1}_{\bar 1},
\
I_rg(y)\mathbf{1}_{\bar 1}=\lambda^{r}(1-2\alpha)g(y+r)\mathbf{1}_{\bar 0};
\\&
\mathrm{or}&
  G_pf(x)\mathbf{1}_{\bar 0}=(-1)^{2p}\lambda^{p}f(y+p)\mathbf{1}_{\bar 1},
\
G_pg(y)\mathbf{1}_{\bar 1}=\lambda^{p}(x+2p\alpha)g(x+p)\mathbf{1}_{\bar 0},
\\&&
I_rf(x)\mathbf{1}_{\bar 0}=2\alpha\lambda^{r}f(x+r)\mathbf{1}_{\bar 1},
\
I_rg(y)\mathbf{1}_{\bar 1}=-\lambda^{r}(1-2\alpha)g(y+r)\mathbf{1}_{\bar 0}.
\end{eqnarray*}
\end{lemm}
\begin{proof}
Let us first consider the first case of   Lemma   \ref{lem4.21452} (i) as  $$G_{0}\mathbf{1}_{\bar0}=\mathbf{1}_{\bar1},  G_{0}\mathbf{1}_{\bar1}
    =x\mathbf{1}_{\bar0}, G_{\frac{1}{2}}\mathbf{1}_{\bar0}=\lambda^{\frac{1}{2}}\mathbf{1}_{\bar1},  G_{\frac{1}{2}}\mathbf{1}_{\bar1}
    =-\lambda^{\frac{1}{2}}(x+\alpha)\mathbf{1}_{\bar0},  \beta=\alpha+\frac{1}{2}.$$
    For $1\neq m\in\Z$, we obtain
\begin{eqnarray*}
G_{m+\frac{1}{2}}\mathbf{1}_{\bar1}
&=&\frac{2}{m-1}\big(L_mG_\frac{1}{2}-G_\frac{1}{2}L_m\big)\mathbf{1}_{\bar1}
\\&=&(-1)^{2m+1}\lambda^{m+\frac{1}{2}}(x+(2m+1)\alpha)\mathbf{1}_{\bar0}.
\end{eqnarray*}
This combine with
$
G_{\frac{3}{2}}\mathbf{1}_{\bar1}
=-\lambda^\frac{3}{2}(x+3\alpha)\mathbf{1}_{\bar0},
$
one has
\begin{eqnarray}\label{gm320}
G_{m+\frac{1}{2}}\mathbf{1}_{\bar1}=(-1)^{2m+1}\lambda^{m+\frac{1}{2}}(x+(2m+1)\alpha)\mathbf{1}_{\bar0},
\end{eqnarray} where $m\in\Z.$
For $0\neq m\in\Z$, it is clear that
 \begin{eqnarray}\label{gm321}
G_{m}\mathbf{1}_{\bar1}
=\frac{2}{m}\big(L_mG_0-G_0L_m)\mathbf{1}_{\bar1}
=\lambda^m(x+2m\alpha)\mathbf{1}_{\bar0}.
\end{eqnarray}
From \eqref{gm320} and \eqref{gm321}, we   conclude
\begin{eqnarray}\label{gp98}
G_{p}f(y)\mathbf{1}_{\bar1}=
(-1)^{2p}\lambda^{p}(x+2p\alpha)f(x+p)\mathbf{1}_{\bar0}
\end{eqnarray} for
$p\in\frac{1}{2}\Z.$

Similarly,   we deduce that
  \begin{eqnarray}\label{gp99}
  G_{p}f(x)\mathbf{1}_{\bar0}
=\lambda^{p}f(y+p)\mathbf{1}_{\bar1}
\end{eqnarray}
for $p\in\frac{1}{2}\Z$.
For any $r\in\frac{1}{2}+\Z$, according to \eqref{gp98} and \eqref{gp99}, we compute that
\begin{eqnarray}\label{ir99}
&&I_{r}f(x)\mathbf{1}_{\bar0}\nonumber
=\frac{1}{r}[G_0,G_r]f(x)\mathbf{1}_{\bar0}
\\&=&\frac{1}{r}\big(\lambda^{r}G_0f(y+r)\mathbf{1}_{\bar1}+G_{r}f(y)\mathbf{1}_{\bar1}\big)\nonumber
\\&=&-2\lambda^{r}\alpha f(x+r)\mathbf{1}_{\bar1},
\\&&I_{r}f(y)\mathbf{1}_{\bar1}=\frac{1}{r}[G_0,G_r]f(y)\mathbf{1}_{\bar1}\nonumber
\\&=&\frac{1}{r}\big(\lambda^{r}G_0(-1)^{2r}(x+2r\alpha)f(x+r)\mathbf{1}_{\bar0}+G_{r}xf(x)\mathbf{1}_{\bar0}\big)\nonumber
\\&=&\lambda^{r}(1-2\alpha)f(y+r)\mathbf{1}_{\bar1}.\label{ir22}
\end{eqnarray}
Then \eqref{gp98}, \eqref{gp99}, \eqref{ir99} and \eqref{ir22} show the first case of the lemma.
From the rest of   Lemma   \ref{lem4.21452} (i) and by  the similar analysis methods, the other case can be obtained.  The lemma holds.
\end{proof}
According to \eqref{eqqqw4.1}, \eqref{eqqew4.2} and Lemma \ref{lem4.33}, a
complete  classification of free $\mathcal{U}(\mathfrak{t})$-modules of rank $2$ over $\T$
 are given.
\begin{theo}\label{th9.5}
Let $W$ be an   $\T$-module such that the restriction of $W$ as a $\mathcal{U}(\mathfrak{t})$-module is free of rank $2$.
 Then up to a
parity, $W\cong\mathcal{N}_1(\lambda,\alpha)$, or  $W\cong\mathcal{N}_{-1}(\lambda,\alpha)$ for $\lambda\in\C^*,\alpha\in\C$.
\end{theo}

\section{Twisted modules}
The following results of the automorphism group of the twisted $N=2$ superconformal algebra appeared in \cite{FHY}.
\begin{prop}
Let $m\in\Z,r\in\frac{1}{2}+\Z,p\in\frac{1}{2}\Z$. Then $\mathrm{Aut}(\T)=\{\omega_b\mid b\in\C^*\}$, where
$$\omega_b(L_m)=-b^{2m}L_{-m},\ \omega_b(I_r)=b^{2r}I_{-r},\ \omega_b(G_p)=b^{2p}\sqrt{-1}G_{-p}.$$
\end{prop}
We have the  statement as follows.
\begin{prop}
Let $m\in\Z,r\in\frac{1}{2}+\Z,p\in\frac{1}{2}\Z,b,\tau\in\C^*,\omega_b\in\mathrm{Aut}(\T),t=\pm1$.
 Assume that    $\mathcal{M}_t(\lambda,\alpha)_b$ is an $\T$-module $\mathcal{M}_t(\lambda,\alpha)$ after a twist by $\omega_b$.
  Then ${{\mathcal{M}_t}}(\lambda,\alpha)\cong\mathcal{M}_t(\tau,\alpha)_b$, where $\tau^p=b^{2p}\lambda^{-p}$.
\end{prop}
\begin{proof}
Take $f(\partial^2)\in\mathcal{M}_t(\lambda,\alpha)_{\bar 0}$ and $\partial g(\partial^2)\in\mathcal{M}_t(\lambda,\alpha)_{\bar 1}$. Let
\begin{eqnarray*}
\Psi:\mathcal{M}_t(\lambda,\alpha)&\longrightarrow&\mathcal{M}_t(\tau,\alpha)_b
\\f(\partial^2)&\longmapsto&f(-\partial^2)
\\  \partial g(\partial^2)&\longmapsto&-\sqrt{-1}\partial g(-\partial^2)
\end{eqnarray*}
be the  linear  map from $\mathcal{M}_t(\lambda,\alpha)$ to $\mathcal{M}_t(\lambda,\alpha)_b$.
It is evident that $\Psi$ is an isomorphic mapping. It follows from
\begin{eqnarray*}
&&\omega_b(L_m)\Psi(f(\partial^2))=\Psi(\omega_b(L_m)f(\partial^2)),
\ \omega_b(L_m)\Psi(\partial g(\partial^2))=\Psi(\omega_b(L_m) \partial g(\partial^2)),
\\&&\omega_b(I_r)\Psi(f(\partial^2))=\Psi(\omega_b(I_r)f(\partial^2)),
\ \omega_b(I_r)\Psi(\partial g(\partial^2))=\Psi(\omega_b(I_r) \partial g(\partial^2)),
\\&&\omega_b(G_p)\Psi(f(\partial^2))=\Psi(\omega_b(G_p)f(\partial^2)),
\ \omega_b(G_p)\Psi(\partial g(\partial^2))=\Psi(\omega_b(G_p)\partial g(\partial^2))
\end{eqnarray*}
 that we obtain
\begin{eqnarray*}&&\omega_b(L_m)f(-\partial^2)=b^{2m}\lambda^{-m}(\partial^2+m\alpha)f(-\partial^2-m),
\\&& \omega_b(L_m)(-\sqrt{-1}\partial g(-\partial^2))=-b^{2m}\lambda^{-m}\sqrt{-1} \partial(\partial^2+m(\alpha+\frac{1}{2}))g(-\partial^2-m),
\\&&\omega_b(I_r)f(-\partial^2)=-(t^{-1}b)^{2r}\lambda^{-r}2\alpha f(-\partial^2-r),
\\&& \omega_b(I_r)(-\sqrt{-1}\partial g(-\partial^2))=-(t^{-1}b)^{2r}\lambda^{-r}(1-2\alpha) \sqrt{-1}\partial g(-\partial^2-r),
\\&&\omega_b(G_p)f(-\partial^2)=(t^{-1}b)^{2p}\lambda^{-p} \partial f(-\partial^2-p),
\\&&  \omega_b(G_p)(-\sqrt{-1}\partial g(-\partial^2))=-(-t^{-1}b)^{2p}\lambda^{-p}\sqrt{-1}(\partial^2+2p\alpha)g(-\partial^2-p)
\end{eqnarray*}
for $m\in\Z,r\in\frac{1}{2}+\Z$ and $p\in\frac{1}{2}\Z.$
Then $\mathcal{M}_t(\tau,\alpha)_b\cong{{\mathcal{M}_t}}(\lambda,\alpha)$, where $\tau^p=b^{2p}\lambda^{-p}, p\in\frac{1}{2}\Z$.  We complete  the proof.
\end{proof}

\section{Intermediate series modules}
We first recall  the definition of   weighting functor, which was introduced by J. Nilsson in \cite{N2}.

Let $\mathfrak{a}$ be a Lie (super)algebra. Suppose that ${\widetilde{\mathfrak{h}}}$   is a
commutative subalgebra of $\mathfrak{a}$ and  $\mathrm{Ad}(\widetilde{\mathfrak{h}})$ acts diagonally on $\mathfrak{a}$, namely,
 $$\mathfrak{a}:=\bigoplus_{\mu\in\widetilde{\mathfrak{h}}^*}\mathfrak{a}_\mu.$$
 Regard each $\mu\in \widetilde{\mathfrak{h}}^*$ as a homomorphism  $\bar \mu:\mathcal{U}(\widetilde{\mathfrak{h}})\rightarrow\C$.
We denote by $\mathrm{Max}(\mathcal{U}(\widetilde{\mathfrak{h}}))$ the set of maximal ideals of $\mathcal{U}(\widetilde{\mathfrak{h}})$.
\begin{defi}\label{wm611}
 For an $\mathfrak{a}$-module $M$, consider the
$\mathcal{U}(\widetilde{\mathfrak{h}})$-module
$$\mathcal{W}(M):=\bigoplus_{m\in \mathrm{Max}(\mathcal{U}(\widetilde{\mathfrak{h}}))}M/mM=\bigoplus_{\mu\in \widetilde{\mathfrak{h}}^*}M/\mathrm{ker}(\bar\mu)M.$$
 Choose $\alpha\in \widetilde{\mathfrak{h}}^*$ and $v\in M$. Define an action of  $x_\alpha\in\mathfrak{a}_\alpha$ on $\mathcal{W}(M)$ by
 $$x_\alpha(v+\mathrm{ker}(\bar \mu)M):=x_\alpha v+\mathrm{ker}(\overline {\mu+\alpha})M.$$
Then the map $\mathcal{W}$ between  $\mathcal{U}(\mathfrak{a})$-mod and $\mathcal{U}(\mathfrak{a})$-mod is called the {\bf weighting functor}.
\end{defi}
In the following, we define the weighting functor $\mathcal{W}$ for $\T$ by taking $\widetilde{\mathfrak{h}}=\mathfrak{t}=\C L_0$.
The weight modules called intermediate series of the twisted $N=2$ superconformal algebra were investigated  in \cite{LSZ}.
\begin{defi}\label{de63} Let $\sigma\in\C,m\in\Z,r\in\frac{1}{2}+\Z,p,k\in\frac{1}{2}\Z,t=\pm1$. {\bf A super intermediate series module} $A_t{(\sigma)}$ over $\T$ is a module
with a basis $\{x_k,y_k\mid k\in\frac{1}{2}\Z\}$ satisfying the  action as follows
\begin{eqnarray*}
&&L_mx_k=(-k+\sigma n)x_{k+m},\ L_my_k=(-k+m(\sigma+\frac{1}{2}))y_{k+m},
\\&& I_rx_k=-2t^{2r}(\sigma+1)x_{k+r},\ I_ry_k=-t^{2r}(2\sigma+1)y_{k+r},
\\&& G_px_k=t^{2p}y_{k+p},\ G_py_k=(-t)^{2p}(-k+(2\sigma+1)p)x_{k+p}.
\end{eqnarray*}
\end{defi}
The following theorem obtained by a very natural extension of    \cite{LSZ}.
\begin{theo}\label{253}
Let $\sigma\in\C,t,t^\prime=\pm1$. The $\T$-module  $A_t{(\sigma)}$ defined as above.
Then
\begin{itemize}
\item[\rm(1)]  the $\T$-module $A_t{(\sigma)}$ is simple if  $\sigma\neq -1,-\frac{1}{2}$;

\item[\rm(2)] $A_t{(\sigma)}\cong A_{t^\prime}{(\sigma^\prime)}$ if and only if $t=t^\prime,\sigma=\sigma^\prime$.
\end{itemize}
\end{theo}
We will describe the main results of this section.
\begin{theo}\label{th664}
Let $\lambda\in\C^*,\alpha\in\C,t=\pm1$. As $\T$-modules, we have $\mathcal{W}(\mathcal{M}_t(\lambda,\alpha))\cong A_t(\alpha-1)$.
\end{theo}
\begin{proof}
Let $\mathcal{M}_t(\lambda,\alpha)=M$ as in Definition \ref{wm611}. For each $n\in\frac{1}{2}\Z$, let $\mu_n\in{\widetilde{\mathfrak{h}}}^*$ be the homomorphism  $L_0=\partial^2\mapsto n.$
We see that $\mathrm{ker}(\overline{\mu_n})=\C[\partial^2](\partial^2-n)$, which is the maximal ideal of $\C[\partial^2].$
Assume that $\mathbf{1}$ is a generator   of $\mathcal{M}_t(\lambda,\alpha)$.
Setting  $$w_n=\mathbf{1}+\mathrm{ker}(\overline{\mu_n})\mathcal{M}_t(\lambda,\alpha)\in\mathcal{M}_t(\lambda,\alpha)_{\bar0}\  \mathrm{and}\ v_n=\partial\mathbf{1}+\mathrm{ker}(\overline{\mu_n})\mathcal{M}_t(\lambda,\alpha)\in\mathcal{M}_t(\lambda,\alpha)_{\bar1},$$
we can  check that
\begin{eqnarray}\label{12qwa}
\nonumber L_mw_{-n}&=&\lambda^m(\partial^2+m\alpha)\mathbf{1}+\mathrm{ker}(\overline{\mu_{-n-m}})\mathcal{M}(\lambda,\alpha)
\\&=&(-n+m(\alpha-1))\lambda^mw_{-(n+m)},
\\\nonumber L_mv_{-n}&=&\lambda^m(\partial^2+m(\alpha+\frac{1}{2}))\partial\mathbf{1}+\mathrm{ker}(\overline{\mu_{-n-m}})\mathcal{M}(\lambda,\alpha)
\\&=&(-n+m(\alpha-\frac{1}{2}))\lambda^mv_{-(n+m)},
\\ I_rw_{-n}&=&-2t^{2r}\lambda^r\alpha\mathbf{1}+\mathrm{ker}(\overline{\mu_{-n-r}})\mathcal{M}(\lambda,\alpha)=-2\alpha t^{2r}\lambda^rw_{-(n+r)},
\\ I_rv_{-n}&=&t^{2r}\lambda^{r}(1-2\alpha)\partial\mathbf{1}+\mathrm{ker}(\overline{\mu_{-n-r}})\mathcal{M}(\lambda,\alpha)=(1-2\alpha)t^{2r}\lambda^rv_{-(n+r)},
\\ G_pw_{-n}&=&t^{2p}\lambda^{p}\partial\mathbf{1}+\mathrm{ker}(\overline{\mu_{-n-p}})\mathcal{M}(\lambda,\alpha)=t^{2p}\lambda^pv_{-(n+p)},
\\\nonumber  G_pv_{-n}&=&(-t)^{2p}\lambda^p(\partial^2+2p\alpha)\mathbf{1}+\mathrm{ker}(\overline{\mu_{-n-p}})\mathcal{M}(\lambda,\alpha)
\\&=&(-t)^{2p}(-n+(2\alpha-1)p)\lambda^pw_{-(n+p)}\label{342qwa}
\end{eqnarray}
for any $m\in\Z,r\in\frac{1}{2}+\Z,p\in\frac{1}{2}\Z$.
We write $\widehat{w}_p=\lambda^pw_{-p},\widehat{v}_p=\lambda^p v_{-p}$ in \eqref{12qwa}-\eqref{342qwa}, one gets
\begin{eqnarray*}
&&L_m\widehat{w}_n=\big(-n+m(\alpha-1)\big)\widehat{w}_{n+m},
\\&& L_m\widehat{v}_n
=\big(-n+m(\alpha-\frac{1}{2})\big)\widehat{v}_{n+m},
\\&& I_r\widehat{w}_n=-2t^{2r}\alpha\widehat{w}_{n+r},
\ I_r\widehat{v}_n=(1-2\alpha)t^{2r}\widehat{v}_{n+r},
\\&& G_p\widehat{w}_n=t^{2p}\widehat{v}_{n+p},
\ G_p\widehat{v}_n=(-t)^{2p}\big(-n+(2\alpha-1)p\big)\widehat{w}_{n+p}.
\end{eqnarray*}
Then we conclude that $\mathcal{W}(\mathcal{M}_t(\lambda,\alpha))\cong A_t(\alpha-1)$.
\end{proof}
\begin{rema}
By   Definition \ref{de63}, it is clear  that  $A_t(\alpha-1)$ is simple if $\alpha\neq0,\frac{1}{2}$. Then  we know that the weighting functor cannot
keep the irreducibility of $\T$-module $\mathcal{M}_t(\lambda,\alpha)$ to  $A_t(\alpha-1)$ for $\alpha=\frac{1}{2}$. From Proposition $10$ of \cite{N2}, we obtain
that $A_t(\alpha-1)$ is a coherent family of degree $1$.
\end{rema}
\begin{coro}
Let $\lambda,\lambda^\prime\in\C^*,\alpha,\alpha^\prime\in\C,t,t^\prime=\pm1.$ If $\alpha\neq\alpha^\prime$ or $t\neq t^\prime$, then we obtain $\mathcal{M}_t(\lambda,\alpha)\ncong \mathcal{M}_{t^\prime}(\lambda^\prime,\alpha^\prime)$.
\end{coro}
\begin{proof}
Suppose  $\mathcal{M}_t(\lambda,\alpha)\cong  \mathcal{M}_{t^\prime}(\lambda^\prime,\alpha^\prime)$.
Then by Theorem \ref{th664} and Theorem \ref{253} (2), one has $\alpha=\alpha^\prime$, $t=t^\prime$, which leads to a contradiction.
\end{proof}

\section*{Acknowledgements}
This work was supported by the National Natural Science Foundation of China (Grant No.11801369).

\bigskip

Haibo Chen
\vspace{2pt}

  School of  Statistics and Mathematics, Shanghai Lixin University of  Accounting and Finance,   Shanghai
201209, China

\vspace{2pt}
hypo1025@163.com

\bigskip

Xiansheng Dai
\vspace{2pt}

 School of Mathematical Sciences, Guizhou Normal University,
Guiyang 550001,  China

\vspace{2pt}

daisheng158@126.com

\bigskip
Mingqiang  Liu
\vspace{2pt}

Three Gorges Mathematics  Research Center, China Three Gorges University, Yichang 443002, China

\vspace{2pt}
mingqiangliu@163.com

\end{document}